\documentclass[12pt]{article}

\usepackage{subfigure}
\usepackage{tikz}

\usetikzlibrary{graphs,graphs.standard,calc}
\usepackage[utf8]{inputenc}

\usepackage{comment}

\usepackage[utf8]{inputenc}
\usepackage[english]{babel}
\usepackage{amsmath}
\usepackage{times}
\usepackage{graphicx}
\usepackage{float}
\usepackage[margin=1in]{geometry}
\usepackage{blindtext}
\usepackage{amssymb}
\usepackage{tikz}
\usepackage{amsthm}
\usepackage{mathtools}
\usepackage{textcmds}

\usepackage{comment}
\usepackage[mathlines]{lineno}
\newtheorem{theorem}{Theorem}[section]
\newtheorem{corollary}[theorem]{Corollary}
\newtheorem{proposition}[theorem]{Proposition}
\newtheorem{lemma}[theorem]{Lemma}
\newtheorem{conjecture}[theorem]{Conjecture}
\theoremstyle{remark}

\theoremstyle{definition}

\renewcommand\footnotemark{}

\title{Extremal values for the spectral radius of the normalized distance Laplacian}
\date{\today}
\author{Jacob Johnston \and Michael Tait\thanks{Villanova University Department of Mathematics \& Statistics. Both authors were partially supported by National Science Foundation grant DMS-2011553 and the second author was partially supported by a Villanova University Summer Grant. Emails: \texttt{$\{$jjohns80, michael.tait$\}$@villanova.edu}}}

\DeclareMathOperator{\diam}{diam}
\DeclareMathOperator{\OBJ}{OBJ}
\begin{document}

\maketitle
\begin{abstract}
    The normalized distance Laplacian of a graph $G$ is defined as $\mathcal{D}^\mathcal{L}(G)=T(G)^{-1/2}(T(G)-\mathcal{D}(G))T(G)^{-1/2}$ where $\mathcal{D}(G)$ is the matrix with pairwise distances between vertices and $T(G)$ is the diagonal transmission matrix. In this project, we study the minimum and maximum spectral radii associated with this matrix, and the structures of the graphs that achieve these values. In particular, we prove a conjecture of Reinhart that the complete graph is the unique graph with minimum spectral radius, and we give several partial results towards a second conjecture of Reinhart regarding which graph has the maximum spectral radius.
\end{abstract}
\section{Introduction}
\par For a graph $G$, one of the most well-studied matrices to associate to $G$ is the normalized Laplacian matrix which is written as $\mathcal{L}(G)=D(G)^{-1/2}(D(G)-A(G))D(G)^{-1/2}$ where $D(G)$ is the diagonal degree matrix, $(D(G))_{ii}=\deg(v_i)$ and $A(G)$ is the adjacency matrix. This matrix was popularized to graph theorists by Chung and has been the subject of much research in part due to its connection with other areas of math, including differential geometry and Markov chains, see for example the monograph \cite{fan}.

The distance matrix, denoted by $\mathcal{D}(G)$, is defined as
\begin{align*}
    (\mathcal{D}(G))_{i j}=d(v_i,v_j)
\end{align*}
where $d(v_i,v_j)$ is the distance between vertex $v_i$ and vertex $v_j$. This matrix was introduced by Graham and Pollak \cite{GP} and was motivated by routing calls in a telephone network. Eigenvalues of distance matrices of graphs were studied extensively after this paper and have received renewed interest in the last decade or so; see the textbook \cite{metricspacebook} and the surveys \cite{AHsurvey, HR}.

In this paper we study a synthesis of the previous two matrices called the {\em normalized distance Laplacian}. This matrix was introduced and studied systematically by Reinhart \cite{R}. To define the matrix we need to define the transmission of a vertex $v_i$ as the sum of distances from $v_i$ to all other $v_j$, that is $t(v_i)=\sum_{i\neq j}d(v_i,v_j)$. We can then define the transmission matrix as the diagonal matrix
\begin{align*}
    (T(G))_{ii}=t(v_i).
\end{align*}
The normalized distance Laplacian of a connected graph $G$ is then defined as
\[
\mathcal{D}^\mathcal{L}(G):=T(G)^{-1/2}(T(G)-\mathcal{D}(G))T(G)^{-1/2}=I-T(G)^{-1/2}\mathcal{D}(G)T(G)^{-1/2},
\]
and has entries
\begin{align*}
     (\mathcal{D}^\mathcal{L}(G))_{i j}= \begin{cases} 
      1 & i=j \\
      -\frac{d(v_i,v_j)}{\sqrt{t(v_i)t(v_j)}} & i\neq j
   \end{cases}.
\end{align*}
When the graph $G$ is clear from context, we will refer to a matrix $M(G)$ associated with it simply as $M$. We denote the eigenvalues of $\mathcal{D}^\mathcal{L}$ as $\partial_1^\mathcal{L}\leq\dots\leq\partial_n^\mathcal{L}$. Let $\mathbf{x}$ be an eigenvector of $\mathcal{D}^\mathcal{L}$. We define the harmonic eigenvector as $\mathbf{y}=T^{-1/2}\mathbf{x}$. Then, if $\mathbf{x}$ is an eigenvector of $T^{-1/2}(T-\mathcal{D})T^{-1/2}$ with eigenvalue $\partial^\mathcal{L}$, we note that $\mathbf{y}$ is eigenvector of $I-T^{-1}\mathcal{D}(G)$ with eigenvalue $\partial^\mathcal{L}$. It follows that any nonzero eigenvalue $\partial^\mathcal{L}$ has harmonic eigenvector $\mathbf{y}\perp T\mathbf{1}$.

\par In her article, Reinhart \cite{R} proved several results about normalized distance Laplacian matrix. In particular she proved the following theorem. 
\begin{theorem}\label{Reinhart bounds theorem}
For a graph $G$ on $n\geq2$ vertices,
\begin{align*}
    \partial_2^\mathcal{L}\leq\frac{n}{n-1}\text{ and }\partial_n^\mathcal{L}\geq\frac{n}{n-1}.
\end{align*}
\end{theorem}
She then proposed the following two conjectures. The graph $KPK_{n_1, n_2, n_3}$ is obtained by connecting two cliques on $n_1$ and $n_3$ vertices together via a path on $n_2$ vertices.
\begin{conjecture}
\label{Reinhart complete}
    For a graph on n vertices,
\begin{align*}
    \partial_n^\mathcal{L}=\frac{n}{n-1}
\end{align*}
if and only if $G$ is the complete graph $K_n$.
\end{conjecture}
\begin{conjecture}
\label{conj2}
    The maximum $\mathcal{D}^\mathcal{L}$ spectral radius achieved by a graph on $n$ vertices tends to 2 as $n\to\infty$ and is achieved by $KPK_{n_1,n_2,n_3}$ for some $n_1+n_2+n_3=n+2$.
\end{conjecture}

In this paper we answer Conjecture \ref{Reinhart complete} affirmatively and give several partial results towards Conjecture \ref{conj2}, including verifying the first part of the statement (Proposition \ref{barbell example}). After completion of this manuscript, we learned that a proof of Conjecture \ref{Reinhart complete} was very recently published in \cite{GRD}, and so our paper gives an alternate proof of the result.  In Section 2, we discuss preliminary work that will lead to the formation of later optimization problems. In Section 3, we prove Conjecture {\ref{Reinhart complete}}. Finally, in Section 4, we show partial results towards Conjecture {\ref{conj2}} via analyzing the optimization problems.

For functions $f,g: \mathbb{N} \to \mathbb{R}_{\geq 0}$ we say that $f = O(g)$ if $\limsup \frac{f}{g} < \infty$, that $f = \Omega(g)$ if $g = O(f)$, and that $f = \Theta(g)$ if both $f = O(g)$ and $f = \Omega(g)$.
\section{Preliminaries}
In this section, we look at the Rayleigh quotient of $\mathcal{D}^\mathcal{L}$ to characterize its eigenvalues as a sum of squares over each pair of vertices.
\begin{proposition}
\label{sum of squares}
Let $\mathbf{x}$ be a nonzero vector with harmonic vector $\mathbf{y}=T^{-1/2}\mathbf{x}$ and let $\mathcal{D}^\mathcal{L}$ be the normalized distance Laplacian matrix of a graph $G$. Then 
\[
\frac{\mathbf{x}^T \mathcal{D}^\mathcal{L} \mathbf{x}}{\mathbf{x}^T\mathbf{x}} = \frac{\sum_{i\not= j} d(v_i,v_j)(\mathbf{y}_i-\mathbf{y}_j)^2 }{\sum_{i=1}^n \mathbf{y}_i^2t(v_i)}.
\]
\begin{proof}
\begin{align*}
\frac{\mathbf{x}^T \mathcal{D}^\mathcal{L} \mathbf{x}}{\mathbf{x}^T\mathbf{x}}&=\frac{\mathbf{x}^T (T^{-1/2}\mathcal{D}^LT^{-1/2}) \mathbf{x}}{\mathbf{x}^T\mathbf{x}}\\
&=\frac{\mathbf{y}^T\mathcal{D}^L\mathbf{y}}{\mathbf{y}^T T^{1/2}T^{1/2}\mathbf{y}}\\
&=\frac{\mathbf{y}^T(T-\mathcal{D})\mathbf{y}}{\sum_{i=1}^n\mathbf{y}_i^2 t(v_i)}\\
&=\frac{\sum_{i,j}\mathbf{y}_j(T_{i j}-d(v_i,v_j))\mathbf{y}_i}{\sum_{i=1}^n\mathbf{y}_i^2 t(v_i)}\\
&=\frac{\sum_{i=j}\mathbf{y}_i^2 t(v_i)-2\sum_{i< j}\mathbf{y}_j d(v_i,v_j)\mathbf{y}_i}{\sum_{i=1}^n\mathbf{y}_i^2 t(v_i)} 
\end{align*}
Then for the first sum in the numerator, we have the term $\mathbf{y}_i^2$ $t(v_i)$ times for each vertex $i$. 
Therefore, we can write $\sum_{i=j}\mathbf{y}_i^2 t(v_i)$ as $\sum_{i\neq j}d(v_i,v_j)(\mathbf{y}_i^2+\mathbf{y}_j^2)$. Hence we have
\begin{align*}
\frac{\mathbf{x}^T \mathcal{D}^\mathcal{L} \mathbf{x}}{\mathbf{x}^T\mathbf{x}}&=\frac{\sum_{i\neq j}d(v_i,v_j)(\mathbf{y}_i^2+\mathbf{y}_j^2)-2\sum_{i\neq j}d(v_i,v_j)\mathbf{y}_i\mathbf{y}_j}{\sum_{i=1}^n\mathbf{y}_i^2 t(v_i)}\\
&=\frac{\sum_{i\neq j}d(v_i,v_j)(\mathbf{y}_i-\mathbf{y}_j)^2}{\sum_{i=1}^n\mathbf{y}_i^2 t(v_i)}
\end{align*}
\end{proof}
\end{proposition}

We note that since the result from Proposition {\ref{sum of squares}} is a sum of squares, we have that any eigenvalue $\partial^\mathcal{L}$ satisfies $\partial^{\mathcal{L}}\geq 0$ with equality holding if and only if all terms are zero, i.e. when $\mathbf{y}=\mathbf{1}$. As $\mathcal{D}^\mathcal{L}$ is a real, symmetric matrix, it admits an orthonormal basis of eigenvectors. Since the vector $T^{1/2}\mathbf{1}$ is an eigenvector corresponding to eigenvalue $0$, the min-max theorem yields the following.
\begin{corollary}
\label{min max}
The second smallest and largest eigenvalues of $\mathcal{D}^\mathcal{L}$ satisfy
\begin{align*}\partial^\mathcal{L}_2&=\min\limits_{\mathbf{y}\perp T\mathbf{1}}\frac{\sum_{i\neq j}d(v_i,v_j)(\mathbf{y}_i-\mathbf{y}_j)^2}{\sum_{i=1}^n\mathbf{y}_i^2 t(v_i)}\\
\partial^\mathcal{L}_n&=\max\limits_{\mathbf{y}\perp T\mathbf{1}}\frac{\sum_{i\neq j}d(v_i,v_j)(\mathbf{y}_i-\mathbf{y}_j)^2}{\sum_{i=1}^n\mathbf{y}_i^2 t(v_i)} = \max\limits_{\mathbf{y}\not=\mathbf{0}}\frac{\sum_{i\neq j}d(v_i,v_j)(\mathbf{y}_i-\mathbf{y}_j)^2}{\sum_{i=1}^n\mathbf{y}_i^2 t(v_i)}\end{align*}
\end{corollary}
Because the eigenvectors are orthogonal, the harmonic eigenvector $\mathbf{y}$ that maximizes $\partial^\mathcal{L}$ automatically satisfies $\sum_i^n\mathbf{y}_i t(v_i)=0$. This allows us to ignore the condition that $\mathbf{y} \perp T\mathbf{1}$ when analyzing optimization problems in Section \ref{optimization section}. Next we write the Rayleigh quotient in a form that is more convenient to work with.
\begin{corollary}
\label{min problem}
Let $\mathbf{y}$ be the harmonic eigenvector of $\mathcal{D}^\mathcal{L}$ corresponding to eigenvalue $\partial^\mathcal{L}_n$. Then $$\partial^\mathcal{L}_n=2-\frac{\sum_{i\neq j}d(v_i,v_j)(\mathbf{y}_i+\mathbf{y}_j)^2}{\sum_i\mathbf{y}_i^2t(v_i)}.$$
\begin{proof}
    From the inequality of arithmetic and geometric means, we have $(\mathbf{y}_i-\mathbf{y}_j)^2\leq 2(\mathbf{y}_i^2+\mathbf{y}_j^2)$, which implies that
     \[
    \frac{\sum_{i\not= j} d(v_i,v_j)(\mathbf{y}_i-\mathbf{y}_j)^2 }{\sum_{i=1}^n \mathbf{y}_i^2t(v_i)}\leq\frac{\sum_{i\not= j} d(v_i,v_j)(2\mathbf{y}_i^2+2\mathbf{y}_j^2)}{\sum_{i=1}^n \mathbf{y}_i^2t(v_i)}= 2.\]
    
    Then,
    \begin{align*}
        \partial^\mathcal{L}_n&\leq 2-\frac{\sum_{i\neq j}d(v_i,v_j)(2\mathbf{y}_i^2+2\mathbf{y}_j^2-(\mathbf{y}_i-\mathbf{y}_j)^2)}{\sum_{i=1}^n \mathbf{y}_i^2t(v_i)}\\
        &=2-\frac{\sum_{i\neq j}d(v_i,v_j)(\mathbf{y}_i+\mathbf{y}_j)^2}{\sum_{i=1}^n \mathbf{y}_i^2t(v_i)}.
    \end{align*}
\end{proof}
\end{corollary}

Note that Corollaries {\ref{min max}} and {\ref{min problem}} yield problems with equivalent solutions. If $\mathbf{y}$ gives the global minimum $z_0$ to the optimization $\min \frac{\sum_{i\neq j}d(v_i,v_j)(\mathbf{y}_i+\mathbf{y}_j)^2}{\sum_{i=1}^n \mathbf{y}_i^2t(v_i)}$, then $\mathbf{y}$ also gives the global maximum $2-z_0$ to the optimization $\max \frac{\sum_{i\neq j}d(v_i,v_j)(\mathbf{y}_i-\mathbf{y}_j)^2}{\sum_{i=1}^n\mathbf{y}_i^2 t(v_i)}$ over all nonzero $\mathbf{y}$. 

\section{Proof of Conjecture \ref{Reinhart complete}}
We use the sum of squares characterization of the eigenvalues to prove Conjecture \ref{Reinhart complete}.
\begin{theorem}Let $G$ be an $n$-vertex graph with $n\geq 2$. Then
$\partial^\mathcal{L}_2=\frac{n}{n-1}$ and $\partial^\mathcal{L}_n=\frac{n}{n-1}$ if and only if $G$ is the complete graph.
\end{theorem}
\begin{proof}
It is straightfoward to check that the eigenvalues of the complete graph satisfy $\partial_2^\mathcal{L} = \cdots = \partial_n^\mathcal{L} = \frac{n}{n-1}$. Assume that $G$ is a graph with $\partial_n^\mathcal{L} = \frac{n}{n-1}$. Then since $\partial_1^\mathcal{L} =0$, we have that  
\[
n = \mathrm{trace}(\mathcal{D}^\mathcal{L}) \leq \sum_{i=2}^n \partial_i^\mathcal{L} \leq (n-1)\partial_n^{\mathcal{L}}.
\]
Therefore, we must have that $\partial_2^\mathcal{L} = \frac{n}{n-1}$ and so the eigenspace corresponding to eigenvalue $\frac{n}{n-1}$ has dimension $n-1$. To prove the conjecture we give a basis for this eigenspace and use the sum of squares characterization to show that the basis implies that all pairs of vertices must be at distance $1$. Since the dimension of the eigenspace is $n-1$, any vector $\mathbf{x}$ which is perpendicular to $T^{1/2}\mathbf{1}$ will give Rayleigh quotient
\[
\frac{\mathbf{x}^T\mathcal{D}^\mathcal{L}\mathbf{x}}{\mathbf{x}^T\mathbf{x}} = \frac{n}{n-1}
\]
Now we can construct linearly independent vectors such that the subspace they create is perpendicular to the vector $T^{1/2}\mathbf{1}$. Without loss of generality, assume that the vertex indexing the first row and column of the matrix has the largest transmission over all vertices (if there is a tie, choose arbitrarily). For $2\leq j\leq n$, define $\mathbf{x}_j$
\begin{align*}
    \mathbf{x}_i= \begin{cases} 
      \sqrt{t(v_j)} & i=1 \\
      -\sqrt{t(v_1)} & i=j \\
      0 & \text{otherwise} 
   \end{cases}
\end{align*}
This vector is perpendicular to $T^{1/2}\mathbf{1}$, and hence it must be an eigenvector for eigenvalue $\frac{n}{n-1}$. Then, in general, the $j^\text{th}$ row of  $\mathcal{D}^\mathcal{L}\mathbf{x}=\partial^\mathcal{L}\mathbf{x}$ will yield
\begin{align*}
    -\frac{d(v_1,v_j)}{\sqrt{t(v_1)t(v_j)}}\sqrt{t(v_j)}-\sqrt{t(v_1)}&=-\partial^\mathcal{L}\sqrt{t(v_1)}.
\end{align*}
Simplifying, we get
\begin{align*}
    \frac{d(v_1,v_j)}{t(v_1)}+1&=\partial^\mathcal{L}.
\end{align*}
This implies that
\begin{align*}
    \frac{d(v_1,v_j)}{t(v_1)}&=\frac{1}{n-1}\text{ for all $j\neq1$.}
\end{align*}

Since $\partial^\mathcal{L}$ is the same for all $n-1$ eigenvectors perpendicular to $T^{1/2}\mathbf{1}$, then the quotient on the left hand side of this equation must be the same for all vertices, including the vertices adjacent to vertex $v_1$, implying that $t(v_1) = n-1$. Then, since this quantity is the same for all $k$ and $t(v_1) = n-1$ for all $j$, we have that $d(v_1,v_j)=1$ for all $j\neq1$. Since $v_1$ has maximum transmission, this implies that for all $j$, $v_j$ has transmission at most $n-1$, so the transmission is exactly $n-1$ for all $j$. Therefore, $G$ is the complete graph.
\end{proof}

\section{Partial Results Towards Conjecture \ref{conj2}}\label{optimization section}

Let $G$ be an $n$-vertex graph with largest normalized distance Laplacian spectral radius $\partial_n^\mathcal{L}$ over all connected $n$ vertex graphs. In this section, we make progress towards Conjecture \ref{conj2} by showing that there exist absolute constants $c_1, c_2, c_3 >0$ satisfying
\[
2 - c_1 \frac{1}{\sqrt{n}}\leq \partial_n^\mathcal{L} \leq 2 - c_2\frac{1}{n},
\]
and that the diameter of $G$ is at least $c_3\sqrt{n}$. We then show that under some natural conditions, we have $\partial_n^\mathcal{L} = 2 - \Theta\left( \frac{1}{\sqrt{n}}\right)$ and the diameter of $G$ is $\Theta(\sqrt{n})$. 
\subsection{Optimization Problems}
In this subsection, we assume that $G$ is an $n$-vertex graph with maximum spectral radius over all graphs on $n$ vertices. We use Corollary {\ref{min problem}} to give partial results towards determining the structure of $G$. First, we define $P$ and $N$ as the sets of vertices with positive and negative harmonic eigenvector entries, respectively. Consider any shortest path of length $\mathrm{diam}(G)$, and define $P'\subset P$ and $N'\subset N$ as the sets of vertices on this shortest path with positive and negative eigenvector entries, respectively. So we have that $|P'|+|N'|=\diam(G)+1$. We are considering the path and the harmonic eigenvector as fixed so that $P$, $N$, $P'$, and $N'$ are fixed sets of indices that are defined by $\mathbf{y}$ and the path. Any harmonic eigenvector entries of $0$ are arbitrarily assigned to $P$ or $N$. From Corollary {\ref{min problem}}, we define the following optimization problem $\mathcal{P}_0$:
\[
    \min\OBJ_0(\mathbf{y}):=\frac{\sum_{i\neq j}d(v_i,v_j)(\mathbf{y}_i+\mathbf{y}_j)^2}{\sum_{i=1}^n \mathbf{y}_i^2t(v_i)}\text{ subject to } \mathbf{y}\neq\mathbf{0}.
\]
If $z_0$ is the minimum of this optimization, then Corollaries \ref{min max} and \ref{min problem} show that $\partial^\mathcal{L}_n = 2-z_0$, and hence we may study Conjecture \ref{conj2} by understanding this optimization. Next we define three more optimization problems which are easier for us to analyze.

First, define $\mathcal{P}_1$, for which we ignore distances:
\[
    \min\OBJ_1(\mathbf{y}):=\frac{\sum_{i\neq j}(\mathbf{y}_i+\mathbf{y}_j)^2}{n\diam(G)\sum_{i=1}^n \mathbf{y}_i^2}
    \text{ subject to } \mathbf{y}\neq\mathbf{0}.
\]
Next, we define optimization $\mathcal{P}_2$ where we ignore some of the terms in the numerator (terms which are close to $0$ in the conjectured extremal example):
\[
    \min\OBJ_2(\mathbf{y}):=\frac{\sum_{v_i,v_j\in P}d(v_i,v_j)(\mathbf{y}_i+\mathbf{y}_j)^2+\sum_{v_i,v_j\in N}d(v_i,v_j)(\mathbf{y}_i+\mathbf{y}_j)^2}{\sum_{i=1}^n\mathbf{y}_i^2t(v_i)}
\]
\[
    \text{subject to } \mathbf{y}_i\geq 0\text{ for all } v_i\in P \text{ and } \mathbf{y}_j\leq 0\text{ for all } v_j\in N, \text{ and } \mathbf{y} \not= \mathbf{0}.
\]
Finally, we define $\mathcal{P}_3$:
\[
    \min\OBJ_3(\mathbf{y}):=\frac{\sum_{i\in P}\frac{|P'|^2}{8}\mathbf{y}_i^2+\sum_{i\in N}\frac{|N'|^2}{8}\mathbf{y}_i^2}{\sum_{i=1}^n\mathbf{y}_i^2t(v_i)}.
\]
\[
    \text{subject to } \mathbf{y}_i\geq 0\text{ for all } v_i\in P \text{ and } \mathbf{y}_j\leq 0\text{ for all } v_j\in N. \text{ and } \mathbf{y} \not= \mathbf{0}.
\]
\begin{lemma}\label{optimization order lemma}
If $z_i$ is a global minimum to $\mathcal{P}_i$, then for $i\geq 1$, $z_0\geq z_i$.
\end{lemma}
\begin{proof}
To begin, we note that for all $\mathcal{P}_i$ as defined above, if $\mathbf{y}$ is in the feasible region of any given optimization problems, then $\mathbf{y}$ is in the feasible region of $\mathcal{P}_0$. First, we look at $\mathcal{P}_1$. We note that in any graph, between any two vertices $v_i$ and $v_j$, $d(v_i,v_j)\geq 1$. Therefore, comparing the numerators of $\OBJ_0$ and $\OBJ_1$, it follows that 
\[\sum_{i\neq j} d(v_i,v_j)(\mathbf{y}_i+\mathbf{y}_j)^2\geq \sum_{i\neq j}(\mathbf{y}_i+\mathbf{y}_j)^2.\]

Then, comparing the denominators of $\OBJ_0$ and $\OBJ_1$, since $t(v_i)\leq n\diam(G)$ for all $v_i$, $\sum_{i=1}^n\mathbf{y}_i^2 t(v_i)\leq n\diam(G)\sum_{i=1}^n\mathbf{y}_i^2$. Thus for any $\mathbf{y}$, we have that $\OBJ_1(\mathbf{y}) \leq \OBJ_0(\mathbf{y})$ and hence the global minima satisfy $z_1\leq z_0$.

Considering the numerator in $\mathcal{P}_2$, we only consider pairs with either both vertices in $P$ or both in $N$. Since we are throwing away nonnegative terms and the denominator remaining the same from $\mathcal{P}_0$, we have that $\OBJ_2(\mathbf{y}) \leq \OBJ_0(\mathbf{y})$ and since the feasible region of $\mathcal{P}_2$ is smaller than that of $\mathcal{P}_0$, we have that $z_2\leq z_0$.

To analyze $\mathcal{P}_3$ we start from $\mathcal{P}_2$. In the numerator of $\OBJ_2(\mathbf{y})$, we have that for each term $(\mathbf{y_i}+\mathbf{y}_j)^2 \geq \mathbf{y}_i^2 + \mathbf{y}_j^2$, as we have thrown away all terms where $\mathbf{y}_i$ and $\mathbf{y}_j$ have different sign. Therefore
\begin{align*}
    \OBJ_2(\mathbf{y})&=\frac{\sum_{v_i,v_j\in P}d(v_i,v_j)(\mathbf{y}_i+\mathbf{y}_j)^2+\sum_{v_i,v_j\in N}d(v_i,v_j)(\mathbf{y}_i+\mathbf{y}_j)^2}{\sum_i^n\mathbf{y}^2t(v_i)}\\
    &\geq\frac{\sum_{v_i\in P}\sum_{v_j\in P}d(v_i,v_j)\mathbf{y}_i^2+\sum_{v_i\in N}\sum_{v_j\in N}d(v_i,v_j)\mathbf{y}_i^2}{\sum_i^n \mathbf{y}_i^2 t(v_i)}
\end{align*}
Now, in the numerator of the second line, each term $\mathbf{y}_i^2$ appears exactly
\[
\sum_{j\in P}d(v_i, v_j) \geq \sum_{j\in P'}d(v_i, v_j) 
\]
times if $i\in P$ and exactly 
\[
\sum_{j\in N}d(v_i, v_j) \geq \sum_{j\in N'}d(v_i, v_j)
\]
times if $i\in N$. For each $i\in P$, since the vertices in $P'$ are on a shortest path, the smallest that $\sum_{j\in P'}d(v_i, v_j)$ can be is if $v_i$ is in the middle of the vertices of $P'$ which are all consecutive on the path. That is,
\[
\sum_{j\in P'}d(v_i, v_j) \geq \sum_{d=1}^{\lfloor (|P'|-1)/2\rfloor} d + \sum_{d=1}^{\lceil (|P'|-1)/2\rceil}d \geq \frac{|P'|^2}{8}. 
\]
Similarly, for all $i\in N$ we have that $\sum_{j\in N'}d(v_i, v_j) \geq \frac{|N'|^2}{8}$. Therefore,
\begin{align*}
\frac{\sum_{v_i\in P}\sum_{v_j\in P}d(v_i,v_j)\mathbf{y}_i^2+\sum_{v_i\in N}\sum_{v_j\in N}d(v_i,v_j)\mathbf{y}_i^2}{\sum_i^n \mathbf{y}_i^2 t(v_i)}\geq \\
 \frac{\sum_{v_i\in P}\frac{|P'|^2}{8}\mathbf{y}_i^2+\sum_{v_i\in N}\frac{|N'|^2}{8}\mathbf{y}_i^2}{\sum_i^n\mathbf{y}_i^2t(v_i)} = \OBJ_3(\mathbf{y}).
\end{align*}
Since the feasible regions of $\mathcal{P}_2$ and $\mathcal{P}_3$ are the same, we have that $z_3\leq z_2\leq z_0$ and $z_1\leq z_0$.
\end{proof}

\subsection{Extremal Graph}
In this subsection, we show a lower bound for $\partial_n^\mathcal{L}$, as well as an absolute upper bound, with conditions that would imply that the upper bound and lower bound are of the same order of magnitude away from $2$. These results come from a determination of upper and lower bounds found for the diameter of the extremal graph $G$.

First, we will find a lower bound for $\partial_n^\mathcal{L}$. Reinhart conjectures that the maximal spectral radius will be achieved by a barbell graph $KPK_{n_1,n_2,n_3}$ with complete cliques of sizes $n_1$ and $n_3$ connected by a path of length $n_2$. We therefore begin by analyzing a specific barbell graph to give our lower bound. Note that since we are only trying to determine the order of magnitude, we ignore floors and ceilings.
\begin{proposition}
\label{barbell example}
For a barbell graph with $\frac{n-\sqrt{2n}}{2}$ vertices in each of the two complete cliques and $\sqrt{2n}$ vertices along the path, then $\partial^\mathcal{L}_n\geq 2-\frac{c}{\sqrt{n}}$ for some absolute constant $c$.
\end{proposition}
\begin{proof}
Let $G$ be the barbell graph with cliques $K_1$ and $K_2$, each with size $k=\frac{n-\sqrt{2n}}{2}$ and path with size $p=\sqrt{2n}$. For vertices $v_i\in K_1,K_2$, $v_i$ has distance $1$ to vertices in its own clique and $p$ to vertices in the other clique. So in total such a vertex has transmission

\[t(v_i)=\sum_{j\in K_1}d(v_i,v_j)+\sum_{j\notin K_1}d(v_i,v_j) = \Omega\left(n^{3/2} \right).\] 
Then, from Proposition {\ref{sum of squares}}, choosing the vector such that $\mathbf{y}_i=1$ for all $v_i\in K_1$, $\mathbf{y}_j=-1$ for all $v_j\in K_2$, and $\mathbf{y}=0$ along the path, we get
\begin{align*}
    \sum_{i\neq j}d(v_i,v_j)(\mathbf{y}_i+\mathbf{y}_j)^2&=1\cdot(2^2)\cdot\frac{k(k-1)}{2}\cdot 2+2k\cdot\left(\sum_{d=1}^p d\right) = O\left( n^2\right)
\end{align*}
Therefore,
\[
\partial_n^\mathcal{L}=2-\frac{\sum_{i\neq j}d(v_i,v_j)(\mathbf{y}_i+\mathbf{y}_j)^2}{\sum_{i=1}^n\mathbf{y}_i^2 t(v_i)}\geq2- cn^{-1/2},
\]
for some absolute constant $c$. A more careful calculation shows that one can take $c = 2^{3/2}+o(1)$.
\end{proof}

The next goal is to find an upper bound for $\partial_n^\mathcal{L}$. Before doing so we need the following technical lemma.

\begin{lemma}\label{vector lower bound}
    Let $\mathbf{z}$ be a unit vector in $\mathbb{R}^n$. Then 
\[
\sum_{i\not=j} (\mathbf{z}_i + \mathbf{z}_j)^2 = \Omega(n).
\]
\end{lemma}
\begin{proof}
    Let $\epsilon >0$ be a small positive constant that will be chosen to be small enough later. As with the eigenvector, we define $P$ and $N$ to be the sets of indices where the vector is positive and negative respectively. That is
    \[
    P = \{i: \mathbf{z}_i \geq 0\} \quad \quad N = \{i: \mathbf{z}_i<0\}.
    \]
    Note that 
    \begin{equation}\label{vector lower bound equation}
\sum_{i\not=j} (\mathbf{z}_i + \mathbf{z}_j)^2 \geq \sum_{i,j\in P}(\mathbf{z}_i + \mathbf{z}_j)^2 + \sum_{i,j \in N} (\mathbf{z}_i + \mathbf{z}_j)^2 \geq (|P|-1)\sum_{i\in P}\mathbf{z}_i^2 + (|N|-1)\sum_{i\in N} \mathbf{z}_i^2.
    \end{equation}
    If $|P|, |N| \geq \epsilon n$, then \eqref{vector lower bound equation} completes the proof. So without loss of generality, assume that $|P| < \epsilon n$ and $|N| \geq (1-\epsilon)n$. With this assumption, if 
    \[
    \sum_{i\in N} \mathbf{z}_i^2 \geq \epsilon,
    \]
    then \eqref{vector lower bound equation} again completes the proof, so we may also assume that 
    \[
    \sum_{i\in P}\mathbf{z}_i^2 \geq 1-\epsilon.
    \]
    Define $N_L = \{i: \mathbf{z}_i < - \sqrt{2\epsilon/n}\}$. Then $|N_L| \leq n/2$ and so $|N\setminus N_L| > (1/2-\epsilon)n > \frac{n}{4}$. Next define $P_S =  \{i: 0\leq\mathbf{z}_i \leq 2\sqrt{2\epsilon/n}\}$. Then 
    \[
    \sum_{i\in P_S} \mathbf{z}_i^2 \leq  |P| (2\sqrt{2\epsilon/n})^2 \leq 8\epsilon^2 < \epsilon,
    \]
    for small enough choice $\epsilon$. This implies that 
    \[
\sum_{i\in P\setminus P_S} \mathbf{z}_i^2 \geq 1-2\epsilon.
    \]
    Now by the definitions of $N_L$ and $P_S$ we have that for any $i\in P\setminus P_S$ and $j\in N\setminus N_L$ we have $\mathbf{z}_i + \mathbf{z}_j \geq \frac{1}{2}\mathbf{z}_i$. Hence
    \[
    \sum_{i\not = j}(\mathbf{z}_i + \mathbf{z}_j)^2 \geq \sum_{\substack{i\in P\setminus P_S\\j\in N\setminus N_L}}(\mathbf{z}_i + \mathbf{z}_j)^2 \geq |N\setminus N_L| \sum_{i\in P\setminus P_S} \frac{\mathbf{z}_i^2}{4} \geq \frac{n}{16}(1-2\epsilon),
    \]
    completing the proof for appropriately chosen $\epsilon$.
\end{proof}

\begin{theorem}
\label{upper bound using diameter}
For the extremal graph $G$, $\partial_n^\mathcal{L}\leq 2-\frac{c}{\diam(G)}$ for some $c>0$.
\end{theorem}

\begin{proof}
  Consider $\OBJ_1(\mathbf{y})$. Since the function is scale-invariant we may assume without loss of generality that $\mathbf{y}$ is a unit vector. By Lemma \ref{vector lower bound}, we have that 
  \[
  \OBJ_1(\mathbf{y}) = \frac{\sum_{i\not=j}(\mathbf{y}_i + \mathbf{y}_j)^2}{n\diam(G)} \geq \frac{c}{\diam(G)},
  \]
  for some $c>0$. By Corollary \ref{min problem} and Lemma \ref{optimization order lemma}, we have the result. 
  \end{proof}

\begin{corollary}
There exists an  absolute constant $c>0$ such that for any graph $G$, $\partial^\mathcal{L}_n\leq 2-\frac{c}{n}$.
\end{corollary}
\begin{proof}
Since for any graph $G$, $\diam(G)\leq n$, from Theorem {\ref{upper bound using diameter}}, it follows that $\partial_n^\mathcal{L}\leq2-\frac{c}{n}$.
\end{proof}

\begin{corollary}
    For the extremal graph, $\diam(G)\geq c\sqrt{n}$ for some $c>0$.
\end{corollary}
\begin{proof}
    From the results in Proposition {\ref{barbell example}} and Theorem {\ref{upper bound using diameter}}, since $G$ is extremal it follows that
    \[
2 - \frac{c_1}{\sqrt{n}} \leq \partial^\mathcal{L}_n \leq 2 - \frac{c_2}{\diam(G)}.
    \] 
     Rearranging, we get $\diam(G)\geq c\sqrt{n}$ for some $c>0$.
\end{proof}

Next, we propose several natural conditions for which if any one is met, then we can upper bound the diameter of the extremal graph $G$, and thus upper bound $\partial_n^\mathcal{L}$. We define the positive transmission of a vertex $v_i$ to be the sum of distances to all vertices $v_j\in P$. Likewise, we define the negative transmission of a vertex $v_i$ to be the sum of distances to all vertices $v_j\in N$. In other words, $t_P(v_i)=\sum_{j\in P}d(v_i,v_j)$ and $t_N(v_i)=\sum_{j\in N}d(v_i,v_j)$.

\begin{theorem}\label{natural conditions}
If any one of the following is true, then for the extremal graph $G$, $\diam(G)\leq c\sqrt{n}$ for some $c>0$ and $\partial_n^\mathcal{L}=2-\Theta\left(\frac{c}{\sqrt{n}}\right)$.
\begin{enumerate}
    \item $\sum_{v_i\in P}\mathbf{y}_i^2,\sum_{v_i\in N}\mathbf{y}_i^2\geq \epsilon\sum_{i}\mathbf{y}_i^2$
    \item $|P'|,|N'|\geq \epsilon \diam(G)$ 
    \item $t_P(v_i),t_N(v_i)\geq \epsilon t(v_i)$ for all $v_i$ and for $\epsilon\geq\frac{c}{\sqrt{n}}$
\end{enumerate}
\end{theorem}
\begin{proof}
    Consider $\OBJ_3(\mathbf{y})$.\\
    Assume that Condition 1 holds: 
    \begin{align*}
        \frac{|P'|^2\sum_{v_i\in P}\mathbf{y}_i^2+|N'|^2\sum_{v_i\in N}\mathbf{y}_i^2}{8\sum_i\mathbf{y}_i^2t(v_i)}&\geq\frac{\epsilon |P'|^2\sum_i\mathbf{y}_i^2+\epsilon |N'|^2\sum_i\mathbf{y}_i^2}{8n\diam(G)\sum_i\mathbf{y}_i^2}\\
        &\geq\frac{\epsilon \diam(G)^2}{32n\diam(G)}\\
        &=\frac{\epsilon \diam(G)}{32n},
    \end{align*}
    where the second inequality is because at least one of $|P'|$ and $|N'|$ is at least $\diam(G)/2$. From Proposition {\ref{barbell example}}, $\frac{\sum_{i\neq j}d(v_i,v_j)(\mathbf{y}_i+\mathbf{y}_j)^2}{\sum_i\mathbf{y}_i^2t(v_i)}\leq\frac{c}{\sqrt{n}}$. Therefore, $\frac{\epsilon \diam(G)}{32n}\leq\frac{c}{\sqrt{n}}$ so $\diam(G)\leq\frac{32c\sqrt{n}}{\epsilon}$\\
    Next assume Condition 2 holds: 
    \begin{align*}
        \frac{|P'|^2\sum_{v_i\in P}\mathbf{y}_i^2+|N'|^2\sum_{v_i\in N}\mathbf{y}_i^2}{8\sum_i\mathbf{y}_i^2t(v_i)}&\geq\frac{2\epsilon^2 \diam(G)^2\sum_i \mathbf{y}_i^2}{8\sum_i\mathbf{y}_i^2t(v_i)}\\
        &\geq\frac{2\epsilon^2 \diam(G)^2\sum_i \mathbf{y}_i^2}{8n\diam(G)\sum_i\mathbf{y}_i^2}\\
        &=\frac{\epsilon^2\diam(G)}{4n}
    \end{align*}
    Similarly to before, by Proposition {\ref{barbell example}}, we have $\frac{\epsilon^2 \diam(G)}{4n}\leq\frac{c}{\sqrt{n}}$ so $\diam(G)\leq\frac{4c\sqrt{n}}{\epsilon^2}$.\\
 Consider $\OBJ_2(\mathbf{y})$ and assume that Condition 3 holds:
    \begin{align*}
    \OBJ_2(\mathbf{y})&=\frac{\sum_{v_i,v_j\in P}d(v_i,v_j)(\mathbf{y}_i+\mathbf{y}_j)^2+\sum_{v_i,v_j\in N}d(v_i,v_j)(\mathbf{y}_i+\mathbf{y}_j)^2}{\sum_i^n\mathbf{y}^2t(v_i)}\\
    &\geq\frac{\sum_{v_i\in P}\sum_{v_j\in P}d(v_i,v_j)\mathbf{y}_i^2+\sum_{v_i\in N}\sum_{v_j\in N}d(v_i,v_j)\mathbf{y}_i^2}{\sum_i^n \mathbf{y}_i^2 t(v_i)}\\
    &=\frac{\sum_{v_i\in P}t_P(v_i)\mathbf{y}_i^2+\sum_{v_i\in N}t_N(v_i)\mathbf{y}_i^2}{\sum_i \mathbf{y}_i^2 t(v_i)}\\
    &\geq\frac{\sum_{v_i\in P}\epsilon t(v_i)\mathbf{y}_i^2+\sum_{v_i\in N}\epsilon t(v_i)\mathbf{y}_i^2}{\sum_i \mathbf{y}_i^2 t(v_i)}\\
    &=\epsilon
\end{align*}
Thus $\partial_n^\mathcal{L}\leq2-\epsilon\leq2-\frac{c}{\sqrt{n}}$
\end{proof}

It remains open to prove Conjecture \ref{conj2}, but Theorem \ref{natural conditions} shows that under any one of three mild conditions on the eigenvector, the extremal graph will have properties similar to the conjectured extremal barbell graph. Showing that in the extremal graph $G$ one of these conditions is true is a natural next step. It seems quite difficult to prove Conjecture \ref{conj2} in full, and it would be interesting even to determine the order of magnitude of the minimum of $2-\partial_n^{\mathcal{L}}$. We end with one more open problem regarding this matrix: how small can $\partial_2^{\mathcal{L}}$ be over all connected $n$ vertex graphs? This is an analog of the same question for the normalized Laplacian, which was asked by Aldous and Fill \cite{AF} and solved in \cite{ACTT}.

\bibliographystyle{plain}
\bibliography{references}

\begin{thebibliography}{1}

\bibitem{ACTT}
Sinan~G Aksoy, Fan Chung, Michael Tait, and Josh Tobin.
\newblock The maximum relaxation time of a random walk.
\newblock {\em Advances in Applied Mathematics}, 101:1--14, 2018.

\bibitem{AF}
David Aldous and James Fill.
\newblock Reversible {M}arkov chains and random walks on graphs, 1995.

\bibitem{AHsurvey}
Mustapha Aouchiche and Pierre Hansen.
\newblock Distance spectra of graphs: a survey.
\newblock {\em Linear algebra and its applications}, 458:301--386, 2014.

\bibitem{fan}
Fan~RK Chung.
\newblock {\em Spectral graph theory}, volume~92.
\newblock American Mathematical Soc., 1997.

\bibitem{metricspacebook}
Michel~Marie Deza, Monique Laurent, and R~Weismantel.
\newblock {\em Geometry of cuts and metrics}, volume~2.
\newblock Springer, 1997.

\bibitem{GRD}
Hilal~A Ganie, Bilal~Ahmad Rather, and Kinkar~Chandra Das.
\newblock On the normalized distance {L}aplacian eigenvalues of graphs.
\newblock {\em Applied Mathematics and Computation}, 438:127615, 2023.

\bibitem{GP}
Ronald~L Graham and Henry~O Pollak.
\newblock On the addressing problem for loop switching.
\newblock {\em The Bell system technical journal}, 50(8):2495--2519, 1971.

\bibitem{HR}
Leslie Hogben and Carolyn Reinhart.
\newblock Spectra of variants of distance matrices of graphs and digraphs: a
  survey.
\newblock {\em La Matematica}, 1(1):186--224, 2022.

\bibitem{R}
Carolyn Reinhart.
\newblock The normalized distance {L}aplacian.
\newblock {\em Special Matrices}, 9(1):1--18, 2021.

\end{thebibliography}
\end{document}